\documentclass[11pt]{amsart}

\usepackage{amsmath}
\usepackage{amssymb}
\usepackage{amsfonts}
\usepackage{graphicx}
\usepackage{mathrsfs}
\usepackage[margin=1.15in]{geometry}

\theoremstyle{plain}
\newtheorem{thm}{Theorem}

\newtheorem{pp}[thm]{Proposition}

\theoremstyle{definition}
\newtheorem{cj}[thm]{Conjecture}
\newtheorem{df}[thm]{Definition}
\newtheorem{ex}[thm]{Example}

\theoremstyle{remark}
\newtheorem{re}[thm]{Remark}

\newcommand{\sF}{\mathcal{F}}
\newcommand{\sI}{\mathcal{I}}
\newcommand{\sJ}{\mathcal{J}}
\newcommand{\sO}{\mathcal{O}}

\newcommand{\CP}{\mathbb{P}}

\newcommand{\reg}{\mathop{\rm reg}\nolimits}
\newcommand{\lct}{\mathop{\rm lct}\nolimits}
\newcommand{\codim}{\mathop{\rm codim}\nolimits}

\begin{document}

\title{Castelnuovo--Mumford regularity\\and log-canonical thresholds}
\author{Alex K\"uronya and Norbert Pintye}
\address{Budapest University of Technology and Economics \\ Mathematical Institute, Department of Algebra \\ P. O. Box 91 \\ H-1521 Budapest, Hungary}
\email{alex.kuronya@math.bme.hu}  
\email{pintye@math.bme.hu}
\subjclass{}
\keywords{}

\begin{abstract}
We prove a sharp inequality relating the Castelnuovo--Mumford regularity of a coherent ideal sheaf to its log-canonical threshold via multiplier ideals. 
\end{abstract}

\maketitle


\setcounter{section}{1}

Ever since the appearance of the influential work~\cite{BM}, it has been a generally accepted principle that the Castelnuovo--Mumford regularity
of a coherent sheaf governs its computational complexity. The main purpose of this short note is to point out how regularity 
influences the log-canonical threshold, an invariant of ideal sheaves arising from birational geometry, which measures 
the singularities of the scheme defined by the ideal in question.

\mbox{}\par\noindent
For a coherent sheaf of ideals $\sI\subseteq\sO_{\CP^n}$, a clean version of the  relation takes the form of the sharp inequality
\[
1\leq\lct(\sI)\reg(\sI) 
\]
as proven in Theorem~\ref{thm:lctreg}. As a by-product of the proof, we formulate a conjecture on the regularity of the integral closure of such ideal sheaves: 
\[
 \reg (\overline{\sI}) \,\leq\, \reg(\sI).
\]

\noindent
We start by briefly recalling some prerequisites. For more thorough information and additional resources, the reader is referred to the appropriate parts of~\cite{PAG1,PAG2}. Until further notice, $X$ is assumed to be a smooth projective variety and $\sI\subseteq\sO_X$ to be a non-zero sheaf of ideals.

\begin{df}
Given an ample basepoint-free divisor $A$ on $X$, we say that a coherent sheaf of $O_X$-modules $\sF$ is \emph{$m$-regular with respect to $A$}\ if
\[H^i(X,\sF\otimes\sO_X((m-i)\cdot A))=0>0.\]
The \emph{Castelnuovo--Mumford regularity}\ of $\sF$ (with respect to $A$), denoted by $\reg_A(\sF)$, is the least integer $m$ for which $\sF$ is $m$-regular, or $-\infty$ if $\sF$ is $m$-regular for all $m\ll 0$. In case of a hyperplane divisor $H$ on $X=\CP^n$, we simply write $\reg(\sF)$ for $\reg_H(\sF)$.
\end{df}

\begin{df}
A \emph{log resolution}\ of the pair $(X,\sI)$ is a projective birational map $\mu:X'\longrightarrow X$ such that $X'$ is non-singular, $\sI\cdot\sO_{X'}=\sO_{X'}(D)$ for an effective divisor $D$, and $D+\mathop{\rm Exc}(\mu)$ has simple normal crossing support. Now, fix such a resolution, and let $c>0$ be a rational number. The \emph{multiplier ideal} $\sJ(c\cdot\sI)$ is defined as \[\sJ(c\cdot\sI)=\mu_*\sO_{X'}(K_{X'/X}-\lfloor c\cdot D\rfloor)\subseteq\sO_X,\] where $K_{X'/X}$ is the relative canonical divisor. It can be shown that this sheaf is independent of the log resolution chosen~\cite[Theorem 9.2.18]{PAG2}. If $c=1$, we usually omit it from the notation.
\end{df}

\begin{thm}[Nadel vanishing]
Let $A$ and $L$ be integral divisors on $X$ such that $\sI\otimes\sO_X(A)$ is globally generated, and $L-c\cdot A$ is big and nef for some rational number $c>0$. Then
\[H^i(X,\sJ(c\cdot\sI)\otimes\sO_X(K_X+L))=0>0.\]
\end{thm}

\noindent For the proof, see~\cite[Corollary 9.4.13]{PAG2}.

\begin{df}
The \emph{log-canonical threshold}\ of $\sI\subset\sO_X$, denoted by $\lct(\sI)$, is defined to be the smallest rational number $c>0$ such that $\sJ(c\cdot\sI)\neq\sO_X$.
\end{df}

\noindent Although the regularity of non-singular varieties is known or expected to be linear in terms of geometric invariants, in certain other cases, especially those of highly singular, it may grow doubly exponentially as a function of the input parameters~\cite[Theorem 3.7]{BM}. However, it is still not known how one can decide which properties influence regularity. It is widely accepted that the smaller the log-canonical threshold, the worse are the singularities~\cite{Kollar}. In the following theorem, we state an inequality that relates these invariants. 

\begin{thm}\label{thm:lctreg}
Let $X=\CP^n$. For a (non-zero) coherent sheaf of ideals $\sI\subseteq\sO_X$,
\[1\leq\lct(\sI)\reg(\sI).\]
\end{thm}

\begin{ex}
The inequality in the theorem is sharp (though not effective -- anyhow, see our notes following the proof). For example, let $\sI=\sO(-H)$ for a hyperplane divisor $H$ on $X$. Then $\lct(\sI)=1$ by Howald's Theorem~\cite{Howald}, and $\reg(\sI)=1$ as well.
\end{ex}

\noindent
The proof goes through a regularity estimate for multiplier ideals. 

\begin{pp}\label{pp:regm}
Let $X=\CP^n$. Given a (non-zero) coherent sheaf of ideals $\sI\subseteq\sO_X$ and a rational number $c>0$, we have the inequality
\[\reg(\sJ(c\cdot\sI))\leq\lfloor c\reg(\sI)\rfloor.\]
In particular, $\reg(\sJ(\sI))\leq\reg(\sI)$.
\end{pp}
\begin{proof}
Suppose that $\sI$ is $m$-regular, and let $H$ be a hyperplane divisor on $X$. Now, on one hand, Mumford's Theorem states that $\sI\otimes\sO(m\cdot H)$ is globally generated, on the other hand, $K_X=-(n+1)\cdot H$. Therefore, we may choose $A=m\cdot H$ and $L$ to be a multiple of $H$, and then apply Nadel's vanishing theorem to obtain an upper bound on $\reg(\sJ(c\cdot\sI))$.

\mbox{}\par
\noindent With this in mind, define $L$ to be the integral divisor $(\lfloor cm\rfloor+1)\cdot H$. Since $H$ is ample, and $(\lfloor cm\rfloor+1)-cm>0$, it follows that $L-c\cdot A$ is ample and, in particular, is big and nef. The higher cohomology groups of $\sJ(c\cdot\sI)\otimes\sO_X(K_X+L)=\sJ(c\cdot\sI)\otimes\sO_X((\lfloor cm\rfloor-n)\cdot H)$ thus vanish. Since $L'=L+k\cdot H$ yields a similar vanishing for all integer $k\geq 0$, it follows that $\sJ(c\cdot\sI)$ is $\lfloor cm\rfloor$-regular, that is, $\reg(\sJ(c\cdot\sI))\leq\lfloor c\reg(\sI)\rfloor$, as we claimed.
\end{proof}

\begin{re}
This proof relies on the fact that the canonical divisor is a multiple of an ample divisor in the case of projective spaces. This is, of course, false in general. However, given an ample divisor $H$ on $X$, we can always find an integer $k$ such that $(k+\varepsilon)\cdot H-K_X$ is ample for every rational number $\varepsilon>0$. Suppose that $H$ is even integral and basepoint-free, and let $L=(k+\lfloor cm\rfloor+1)\cdot H-K_X$ with $k$ accordingly. Then, proceeding in the same way as before, one can show that
\[\reg_H(\sJ(c\cdot\sI))\leq k+\lfloor c\reg_H(\sI)\rfloor+1+\dim X.\]
\end{re}

\begin{proof}[Proof of Theorem~\ref{thm:lctreg}]
Since we are working over $X=\CP^n$, the regularity of an ideal sheaf $\sI\subseteq O_X$ is always non-negative. Moreover, it is zero if and only if $\sI=O_X$. Let $c=\lct(\sI)$. Then, by definition, $\sJ(c\cdot\sI)\neq\sO_X$, therefore, using Proposition~\ref{pp:regm}, our claim immediately follows from
\[1\leq\reg(\sJ(c\cdot\sI))\leq\lfloor c\reg(\sI)\rfloor\leq\lct(\sI)\reg(\sI).\]
\end{proof}

\begin{re}
In the proof of Proposition~\ref{pp:regm}, we use only that $\sI\otimes\sO(m\cdot H)$ is globally generated. If we let $m=d(\sI)$ be the least such integer, the inequality of Theorem~\ref{thm:lctreg} can be improved to 
\[
1\leq\lct(\sI)\mskip1mu d(\sI).
\]
Though this is weaker than the inequality $\codim(\sI)\leq\lct(\sI)\mskip1mu d(\sI)$ proven in~\cite[Corollary 3.6]{FEM} using jet schemes, it turns out that a minor modification of our approach gives  lower bounds that are much more precise than the one coming from~\cite{FEM}. 

\mbox{}\par
\noindent Let $c=\lct(\sI)$. First, note that $\codim(\sI)\leq\reg(\sJ(c\cdot\sI))$ cannot hold in general. Consider for instance $\sI=\widetilde{I}$ with $I=(x_0^2,x_1x_2^2)\subset\mathbb{C}[x_0,x_1,x_2]$; then $\codim(\sI)=2$, while $\reg(\sJ(c\cdot\sI))=1$. However, the following minor modification of our method leads to more precise estimates. Let $c_k=kc$ for $k>0$. Then
\[
\limsup_{k>0}\frac{\reg(\sJ(c_k\cdot\sI))}{k}\leq c\cdot d(\sI)=\lct(\sI)\mskip1mu d(\sI).
\]
Going back to our example, if we compute the fraction for $k=2$, we can already see that the limit superior is at least $2$, immediately yielding the same bound as the codimension. On the other hand, the bound coming from multiplier ideals is often much stronger than the codimension. For a simple illustration, take $I=(x_0\cdots x_n)\subset\mathbb{C}[x_0,\ldots,x_n]$. Then, we have 
\[
\codim(\sI)=1,\ \text{whereas }\reg(\sJ(c\cdot\sI))=c\cdot d(\sI)=n+1.  
\]
\end{re}


One of the important algebraic properties of multiplier ideals is that they are integrally closed being, in fact, very close
to the integral closure of the ideal given. The regularity estimate for multiplier ideals suggests that a similar result might hold for integral
closures. 

\begin{df}\label{def:intcl}
Let $\nu:Y\longrightarrow X$ be the normalised blow-up of $\sI$ with (a not necessarily irreducible) exceptional divisor $E=\mathop{\rm Exc}(\nu)$, so that $\sI\cdot\sO_Y=\sO_Y(-E)$. The \emph{integral closure}\ of $\sI$ is then defined to be the ideal sheaf
\[\overline{\sI}=\nu_*\sO_Y(-E)\subseteq\overline{\sO_X}=\sO_X.\]
We have the inclusion $\sI\subseteq\overline{\sI}$, and in case of equality, we say that $\sI$ is integrally closed.
\end{df}

\begin{re}
It can be shown that $\sJ(c\cdot\overline{\sI})=\sJ(c\cdot\sI)$ for all rational numbers $c>0$ by taking the normalised blow-up as the first step in a log resolution. Moreover, $\overline{\sI}\subseteq\sJ(\sI)$. Thus, based on Proposition~\ref{pp:regm}, the question is then how $\reg(\overline{\sI})$ and $\reg(\sI)$ are related. Note that, for any ideal $I\subseteq\mathbb{C}[x_0,\ldots,x_n]$, there is an algebraic way to define its integral closure $\overline{I}$, and if $I$ is homogeneous, then we have the corresponding notion of its regularity $\reg(I)$ as well. Now, if $X=\CP^n$ and $\sI=\widetilde{I}$ with $I$ saturated, then $\reg(I)=\reg(\sI)$ and $\overline{I}=\bigoplus_{d\geq 0}H^0(X,\overline{\sI}\otimes\sO_X(d))$.{\parfillskip=0pt\par}
\end{re}

\begin{cj}\label{conj:RegularityOfIntegralClosure}
Let $X=\CP^n$. If a (non-zero) coherent sheaf of ideals $\sI\subseteq\sO_X$ is $m$-regular, then so is $\overline{\sI}$. In particular, $\reg(\overline{\sI})\leq\reg(\sI)$.
\end{cj}

\begin{re}
Our conjecture admits a generalisation to integral extensions, which, as pointed out in~\cite{VP}, can be completely described by the corresponding inverse image ideal sheaves under the normalised blow-up. More precisely, we have that $\sI\subseteq\sJ$ is an integral extension if and only if $\sI\cdot\sO_Y=\sJ\cdot\sO_Y$. Now, the conjecture reads as follows: if a (non-zero) coherent sheaf of ideals is $m$-regular, then so is any of its integral extensions.
\end{re}

\pagebreak
\begin{re}
For some time, it had been an open question whether $\reg(\sqrt{I})\leq\reg(I)$. The answer given in~\cite{Chardin} was negative -- taking the radical of an ideal may increase the regularity. However, the ideals (in fact, families of ideals) in question are complete intersections, and thus integrally closed. The following example, therefore, might be of interest.
\end{re}

\begin{ex}
Let $S=\mathbb{C}[x_0,x_1,x_2,x_3,x_4]$, and consider the saturated homogeneous ideal
\[I=(x_0^2-x_4^2,x_1^2-x_4^2,x_0x_3^3-x_2^3x_4)\subseteq S.\]
Then $\reg(I)=6$, $\reg(\sqrt{I})=7$, and $\reg(\overline{I})=5$.
\end{ex}

\begin{re}
Due to Brian\c{c}on--Skoda, the regularity functions of $\sI$ and $\overline{\sI}$ (that is, $a\mapsto\reg(\sI^a)$) should be pretty close, even asymptotically the same. Indeed, this is a nice consequence of the theory of $s$-invariants. Though these functions need not be linear, since $s(\sI)=s(\overline{\sI})$, they have the same slope~\cite[Theorem B and Proposition 1.11]{CEL}.
\end{re}

\subsubsection*{Acknowledgements}
Alex K\"uronya was partially supported by the DFG-Forschergruppe 790 ``Classification of Algebraic Surfaces and Compact Complex Manifolds'', by the DFG-Graduier\-ten\-kol\-leg 1821 ``Cohomological Methods in Geometry'', and by the OTKA grants 77476 and 81203 of the Hungarian Academy of Sciences.\par
\mbox{}\par
\noindent Norbert Pintye was partially supported by the grants T\'AMOP-4.2.1/B-09/1/KMR-2010-0002 and T\'AMOP-4.2.2/B-10/1-2010-0009.
\mbox{}\par
\mbox{}\par
\noindent Helpful discussions with Lawrence Ein are gratefully acknowledged.


\bibliographystyle{plain}
\bibliography{article}

\end{document}